\documentclass[12pt]{amsart}

\usepackage{amssymb}
\usepackage{bm}
\usepackage{graphicx}
\usepackage{psfrag}
\usepackage[usenames,dvipsnames]{xcolor}
\usepackage{enumerate}
\usepackage{multirow}
\usepackage{url}

\usepackage{comment}

\usepackage{algpseudocode}

\usepackage{mathtools}

\usepackage{xy}
\input xy
\xyoption{all}

\newcommand{\excise}[1]{}

\newtheorem{thm}{Theorem}[section]
\newtheorem{lemma}[thm]{Lemma}

\theoremstyle{definition}

\newtheorem{example}[thm]{Example}
\newtheorem{remark}[thm]{Remark}

\newtheorem{notation}[thm]{Notation}

\numberwithin{equation}{section}

\renewcommand\>{\rangle}
\newcommand\<{\langle}

\newcommand\ZZ{\mathbb{Z}}

\DeclareMathOperator\Betti{Betti} 

\begin{document}

\mbox{}
\title[Minimal presentations of shifted affine semigroups]{On minimal presentations of shifted affine semigroups with few generators}

\author[O'Neill]{Christopher O'Neill}
\address{Mathematics Department\\San Diego State University\\San Diego, CA 92182}
\email{cdoneill@sdsu.edu}

\author[White]{Isabel White}
\address{Mathematics Department\\San Diego State University\\San Diego, CA 92182}
\email{iwhite1202@sdsu.edu}

\date{\today}

\begin{abstract}
An affine semigroup is a finitely generated subsemigroup of~$(\ZZ_{\ge 0}^d, +)$, and a numerical semigroup is an affine semigroup with $d = 1$.  A growing body of recent work examines shifted families of numerical semigroups, that is, families of numerical semigroups of the form $M_n = \<n + r_1, \ldots, n + r_k\>$ for fixed $r_1, \ldots, r_k$, with one semigroup for each value of the shift parameter $n$.  
It has been shown that within any shifted family of numerical semigroups, the size of any minimal presentation is bounded (in fact, this size is eventually periodic in $n$).  In this paper, we consider shifted families of affine semigroups, and demonstrate that some, but not all, shifted families of 4-generated affine semigroups have arbitrarily large minimal presentations.  
\end{abstract}

\maketitle


\section{Introduction}
\label{sec:intro}

An \emph{affine semigroup} $M$ is a finitely generated subsemigroup of~$(\ZZ_{\ge 0}^d, +)$, and we call~$d$ the \emph{affine dimension} of $M$.  If $d = 1$, we call $M$ a \emph{numerical semigroup}.  We write 
$$M = \<v_1, \ldots, v_k\> = \{a_1v_1 + \cdots + a_kv_k : a_1, \ldots, a_k \in \ZZ_{\ge 0}\} \subset \ZZ_{\ge 0}^d$$
to specify the affine semigroup with \emph{generating set} $v_1, \ldots, v_k \in \ZZ_{\ge 0}^d$.  Each expression 
$$v = a_1v_1 + \cdots + a_kv_k \qquad \text{with} \qquad a_1, \ldots, a_k \in \ZZ_{\ge 0}$$
of an element $v \in M$ as a sum of generators of $M$ is called a \emph{factorization} of $v$, which we often represent by the $k$-tuple $(a_1, \ldots, a_k)$.  

One of the primary ways of studying an affine semigroup $M$ is via a \emph{minimal presentation} $\rho \subset \ZZ_{\ge 0}^k \times \ZZ_{\ge 0}^k$, each element of which is a pair of factorizations that represents a minimal \emph{relation} or \emph{trade} between the generators of $M$ (we defer the formal definition of minimal presentations until Section~\ref{sec:background}).  As an example, if $M = \<6, 9, 20\>$ is the \emph{Chicken McNugget semigroup}~\cite{mcnuggetmag}, then 
$$\rho = \big\{ \! \big( (3,0,0), (0,2,0) \big), \big( (4,4,0), (0,0,3) \big) \! \big\}$$
is one possible minimal presentation for $M$, where the first element represents the relation $3 \cdot 6 = 2 \cdot 9$ and the second represents the relation $4 \cdot 6 + 4 \cdot 9 = 3 \cdot 20$.  

This paper considers parametrized families of affine semigroups of the form
$$M_n = \<f_1(n), \ldots, f_k(n)\> \subset \ZZ_{\ge 0}^d$$
for some functions $f_1(n), \ldots, f_k(n)$ whose coordinates are polynomials in an integer parameter $n$.  A growing body of recent work~\cite{presburgerarith,shiftydelta,shiftedtangentcone,shiftyapery,shifted3gen} examines the asymptotic behavior of various combinatorially-flavored semigroup invariants, viewed as functions of $n$, and numerous invariants have been shown for large $n$ to coincide with a \emph{quasipolynomial}, that is, a polynomial whose coefficients are periodic functions of $n$.  

Recent work of Bogart, Goodrick and Woods~\cite{parametricpresburgerbigthm} prove that for any parametrized family of affine semigroups, the function $n \mapsto |\rho_n|$ is eventually quasipolynomial in~$n$, where each $\rho_n$ denotes a minimal presentation of $M_n$.  Their results are broad, but nonconstructive, relying on techniques from formal logic known as Presburger arithmetic.  In particular, combinatorial details of the quasipolynomial function, such as its degree, period, and (often constant) leading coefficient, are not known in general.  

Parametrized families of numerical semigroups have recieved substantially more attention.  When $M_n$ is a numerical semigroup of the form
$$M_n = \<n + r_1, \ldots, n + r_k\> \subset \ZZ_{\ge 0}$$
for fixed $r_1, \ldots, r_k \in \ZZ_{\ge 0}$ (which we call a \emph{shifted} family of numerical semigroups), it~is known that the function $n \mapsto |\rho_n|$ is eventually periodic~\cite{shiftyminpres,vu14}.  More generally, if each generator instead has the form $w_in + r_i$ for some $w_i, r_i \in \ZZ_{\ge 0}$, then it is again known that $n \mapsto |\rho_n|$ is eventually periodic~\cite{kerstetter}.  

In this paper, we provide an initial investigation into the combinatorial details of the function $n \mapsto |\rho_n|$ for families of affine semigroups.  In particular, we consider affine semigroups of the form 
$$M_n = \<(n, n) + (x_1, y_1), \ldots, (n, n) + (x_k, y_k)\> \subset \ZZ_{\ge 0}^2$$
for fixed $x_i, y_i \in \ZZ_{\ge 0}$ (which we call \emph{shifted affine semigroups}).  We~prove that if $k = 3$, then $|\rho_n|$~is eventually periodic (Theorems~\ref{t:3gen1} and~\ref{t:3gen2}).  In contrast, for $k = 4$, we provide examples demonstrating that $|\rho_n|$ grows unbouned for some shifted families (Theorem~\ref{t:4genlinear}), but is eventually periodic for others (Theorem~\ref{t:4genperiodic}).

\section{Background}
\label{sec:background}

Fix an affine semigroup $M = \<v_1, \ldots, v_k\> \subset \ZZ_{\ge 0}^d$.  In what follows, we develop the notion of a minimal presentation; for a more detailed introduction, see~\cite{fingenmon,numerical}.  

The \emph{factorization homomorphism}
$$\begin{array}{r@{}c@{}l}
\pi:\ZZ_{\ge 0}^k &{}\longrightarrow{}& M \\
z &{}\longmapsto{}& z_1v_1 + \cdots + z_kv_k
\end{array}$$
is the additive semigroup homomorphism that sends each $k$-tuple $z = (z_1, \ldots, z_k)$ to the element of $M$ that $z$ is a factorization of.  Under this notation, the preimage $\pi^{-1}(v)$ is the \emph{set of factorizations} of $v \in M$.  

A \emph{relation} or \emph{trade} of $M$ is a pair $(z, z')$ of factorizations such that $\pi_n(z) = \pi(z')$.  The set $\ker\pi \subset \ZZ_{\ge 0}^k \times \ZZ_{\ge 0}^k$ of relations of $M$, given by
$$\ker\pi = \{(z, z') : \pi(z) = \pi(z')\},$$
is an equivalence relation on $\ZZ_{\ge 0}^k$ that is additionally closed under \emph{translation}:\ whenever $(z, z') \in \ker\pi$, we have $(z + u, z' + u) \in \ker\pi$ for any $u \in \ZZ_{\ge 0}^k$.  
This makes $\ker\pi$ a \emph{congruence} on $\ZZ_{\ge 0}^k$, called the \emph{kernel congruence} of $\pi$.  

A \emph{presentation} for $M$ is a subset $\rho \subset \ker\pi$ such that $\ker\pi$ is the smallest congruence on $\ZZ_{\ge 0}^k$ containing $\rho$ (or, equivalently, if $\ker\pi$ equals the intersection of all congruences on~$\ZZ_{\ge 0}^k$ containing $\rho$).  
We say a presentation $\rho$ of $M$ is \emph{minimal} if no proper subset of~$\rho$ is a presentation for $M$.  Although a given affine semigroup $M$ can have several distinct minimal presentations, it is known that all minimal presentations of $M$ are finite and contain the same number of relations.  

The \emph{Betti elements} of $M$ are those in the set
$$\Betti(M) = \{\pi(z) : (z, z') \in \rho\},$$
where $\rho$ is any minimal presentation of $M$ (it is known that the set $\Betti(M)$ is independent of the choice of $\rho$).  
Given an element $v \in M$, we define a graph $\nabla_v$ whose vertex set is $\pi^{-1}(v)$ and where two vertices $z, z' \in \pi^{-1}(v)$ are connected by an edge whenever $z_i > 0$ and $z_i' > 0$ for some $i$.  It turns out $v \in \Betti(M)$ if and only if $\nabla_v$ is disconnected.  Moreover, the number of relations $(z, z') \in \rho$ for which $v = \pi(z)$ is one less than the number of connected components of $\nabla_v$.  This connection between Betti elements and minimal relations will play a key role in Theorem~\ref{t:4genlinear}.  

\begin{example}\label{e:raisingcanes}
The affine semigroup 
$$C = \<(3,2), (4,3), (6,3)\>$$
is the \emph{Raising Cane's semigroup}, named for the popular southern fried chicken restaurant.  Raising Cane's has the following combos: 
\begin{enumerate}[(i)]
\item 
the 3-Finger Combo, which comes with 3 chicken fingers and 2 sides; 

\item 
the Box Combo, which comes with 4 chicken fingers and 3 sides; and 

\item 
the Caniac Combo, which comes with 6 fingers and 3 sides.  

\end{enumerate}
As such, a vector $(x, y)$ lies in $C$ if you can purchase exactly $x$ chicken fingers and $y$ sides using the combo boxes listed above. 

The Raising Cane's semigroup has minimal presentation
$$\rho = \big\{ \! \big( (6,0,0), (0,3,1) \big) \! \big\}$$
since purchasing $6$ of the $3$-Finger Combos yields the same number of chicken fingers and sides as purchasing $3$ of the Box Combos and $1$ Caniac Combo.  
\end{example}

It turns out that, like the semigroup in Example~\ref{e:raisingcanes}, any 3-generated affine semigroup in $\ZZ_{\ge 0}^2$ has a unique minimal presentation consisting of a single relation, a fact we record here for use in our proofs of Theorems~\ref{t:3gen1} and~\ref{t:3gen2}.  

\begin{lemma}\label{l:3genminpres}
Any 3-generated affine semigroup $S \subset \ZZ_{\ge 0}^2$ has a unique trade $(z, z')$ with $|z| < |z'|$ in which the coordinates of $z$ and $z'$ do not all have a common factor.  Moreover, $\{(z, z')\}$ is the unique minimal presentation of $S$.  
\end{lemma}

\begin{proof}
This follows immediately from \cite[Corollary 1.6]{completeintersection}.
\end{proof}

We close the section with a result that first appeared as \cite[Theorem~4.9]{shiftyminpres}, which implies that within any family of shifted numerical semigroups $M_n$ with corresponding minimal presentations $\rho_n$, the map $n \mapsto |\rho_n|$ is eventuall periodic in $n$.  

\begin{thm}\label{t:numericalshifting}
Fix $r_1, \ldots, r_k \in \ZZ_{\ge 0}$, and let $p = r_k - r_1$.  For each $n \in \ZZ_{\ge 1}$, let
$$M_n = \<n + r_1, \ldots, n + r_k\>,$$
and let $\pi_n:\ZZ_{\ge 0}^k \to M_n$ denote the factorization homomorphism of $M_n$.  Consider the map $\Phi_n:\ker \pi_n \to \ker \pi_{n+p}$ given by 
$$(z, z') \mapsto \begin{cases}
(z + \ell e_k, z' + \ell e_1) & \text{if $|z| < |z'|$;} \\
(z + \ell e_1, z' + \ell e_k) & \text{if $|z| > |z'|$;} \\
(z, z') & \text{if $|z| = |z'|$,}
\end{cases}$$
where $\ell = \big| |z| - |z'| \big|$.  
If $n > p^2$, then $\Phi_n$ sends any minimal presentation of $M_n$ to a minimal presentation of $M_{n+p}$.  
\end{thm}

\section{Shifted affine semigroups with 3 generators}
\label{sec:3gen}

The main results of this section are Theorems~\ref{t:3gen1} and~\ref{t:3gen2}, which together establish an analogous result to Theorem~\ref{t:numericalshifting} for 3-generated affine semigroups in affine dimension~$2$.  

\begin{notation}\label{n:3gen}
Throughout this section, let $r_1 = (x_1, y_1), r_2 = (x_2, y_2) \in \ZZ_{\ge 0}^2$ be fixed, and for each $n \in \ZZ_{\ge 1}$ define the affine semigroup
$$M_n = \<N, N + (x_1, y_1), N + (x_2, y_2)\>$$
and the corresponding factorization homomorphism $\pi_n:\ZZ_{\ge 0}^3 \to \ZZ_{\ge 0}^2$ given by
$$\pi_n(z_0, z_1, z_2) = z_0N + z_1(N + r_1) + z_2(N + r_2) = |z|N + z_1r_1 + z_2r_2,$$
where $N = (n,n)$.  Additionally, define 
$$
a_1 = |x_1 - y_1|,
\quad
a_2 = |x_2 - y_2|,
\quad
d_1 = \frac{a_1}{\gcd(a_1,a_2)},
\quad
d_2 = \frac{a_2}{\gcd(a_1,a_2)},
$$
and
$$p = \frac{|x_1y_2 - y_1x_2|}{\gcd(a_1,a_2)}.$$
\end{notation}

\begin{remark}\label{r:3gengeometry}
The precise relationship between minimal presentations within the same shifted family is dependent upon the geometric orientation of the generators of $M_n$.  In~particular, there exist two possible orientations of generators of $M_n$.  
\begin{enumerate}[(a)]
\item 
If $x_1 < y_1$ and $x_2 > y_2$ (or if $x_2 < y_2$ and $x_1 > y_1$), then $N + (x_1, y_1)$ and $N + (x_2, y_2)$ can be oriented such that they are on either side of $N$.  These conditions can be consolidated by requiring $(x_1 - y_1)(x_2 - y_2) < 0$. 

\item 
If $x_1 > y_1$ and $x_2 > y_2$ (or if $x_1 < y_1 $ and $x_2 < y_2$), then $N + (x_1, y_1)$ and $N + (x_2, y_2)$ can be oriented such that they are on one side of $N$.  As before, these conditions can be consolidated to $(x_1 - y_1)(x_2 - y_2) > 0$.  

\end{enumerate}
Theorems~\ref{t:3gen1} and~\ref{t:3gen2}, respectively, handle these two cases.  
\end{remark}

\begin{thm}\label{t:3gen1}
Suppose $(x_1-y_1)(x_2 - y_2) \le 0$, and let $\Psi_n: \ker{\pi_n} \rightarrow \ker{\pi}_{n+p}$ given~by 
\begin{center}
$\Psi_n(z,z') = \begin{cases}
\big( z + \ell(d_2e_1 + d_1e_2), z' + \ell(d_1+d_2)e_0 \big) & \text{if $|z| < |z'|$;} \\
\big( z + \ell(d_1+d_2)e_0, z' + \ell(d_2e_1 + d_1e_2) \big) & \text{if $|z| > |z'|$;} \\
(z,z') & \text{if $|z| = |z'|$,}
\end{cases}$
\end{center}
where $\ell = \big| |z| - |z'| \big|$.  If $\{(z,z')\}$ is the minimal presentation of $M_n$, then $\{\Psi_n(z,z')\}$ is the minimal presentation of $M_{n+p}$.  
\end{thm}

\begin{proof}
By symmetry, we can assume without loss of generality that $a_1 = y_1 - x_1 \ge 0$ and that $a_2 = x_2 - y_2 \ge 0$.  As such, 
$$p = \frac{x_2y_1 - x_1y_2}{\gcd(a_1, a_2)}$$
and for any $(z, z') \in \ker \pi_n$,
$$\Psi_n(z,z') = \big( z + \ell(d_2e_1 + d_1e_2), z' + \ell(d_1+d_2)e_0 \big).$$
We must first show $\Psi_n$ is well-defined, that is, whenever $(z,z') \in \ker \pi_n$, we have $(w, w') := \Psi_n(z,z') \in \ker \pi_{n+p}$.  By symmetry, it suffices to consider the case $|z| \le |z'|$.  By assumption, $\pi_n(z) = \pi_n(z')$, so
\begin{align*}
w - w'
&= \pi_{n+p} \big( z + \ell(d_2e_1 + d_2e_2)) - \pi_{n+p}(z' + \ell(d_1+d_2)e_0 \big) \\
&=  \big( (|z| + \ell(d_1+d_2))(n+p, n+p) + (z_1+ \ell d_2)r_1 + (z_2 + \ell d_1)r_2 \big) \\
& \qquad - \big( (|z'| + \ell(d_1+d_2))(n+p, n+p) + z_1'r_1 + z_2'r_2 \big) \\
&=  (|z| - |z'|)(n+p, n+p) + (z_1 - z_1' + \ell d_2)r_1 + (z_2 - z_2' + \ell d_1)r_2 \\
&= \pi_n(z) - \pi_n(z') + \ell(p, p) + \ell(d_2r_1 - d_1r_2) \\
&= \ell (p, p) +  \ell \bigg( \! \dfrac{(x_2-y_2)x_1}{\gcd(a_1,a_2)} - \dfrac{(y_1 - x_1)x_2}{\gcd(a_1,a_2)}, \dfrac{(x_2-y_2)y_1}{\gcd(a_1,a_2)} - \dfrac{(y_1-x_1)y_2}{\gcd(a_1,a_2)} \! \bigg) \\
&= \ell (p, p) - \ell \bigg( \! \dfrac{x_2y_1 - y_2x_1}{\gcd(a_1,a_2)}, \dfrac{x_2y_1 - y_2x_1}{\gcd(a_1,a_2)} \! \bigg) = (0,0),
\end{align*} 
thereby proving $\Psi_n$ is well-defined.  

By Lemma~\ref{l:3genminpres}, it remains to show that if the coordinates of $(z,z')$ have no common factors, then the coordinates of $(w, w') = \Psi_n(z,z')$ has no common factors.  To that end, suppose $k \in \ZZ_{\ge 1}$ divides every coordinate in both $w$ and $w'$.  We must show that $k$ divides every coordinate of $z$ and $z'$.  This means $k \mid z_i$ for $i = 0, 1, 2$, meaning $k$ divides $|w| = w_0 + w_1 + w_2$.  Likewise, $k$ divides $|w'|$, and therefore $k$ divides 
$$\ell = \big| |w| - |w'| \big| = \big| |z| - |z'| \big|.$$
As such, $k$ divides every coordinate of $z = w - \ell(d_2e_1 + d_1e_2)$ as well as every coordinate of $z' = w' - \ell(d_1 + d_2)e_0$.  This completes the proof.  
\end{proof}

The proof of Theorem~\ref{t:3gen2} is analogous to that of Theorem~\ref{t:3gen1}, and thus is omitted.  

\begin{thm}\label{t:3gen2}
If $(x_1 - y_1)(x_2 - y_2) > 0$, then the map $\Psi_n': \ker{\pi_n} \rightarrow \ker{\pi}_{n+p}$ given by 
\begin{center}
$\Psi'_n(z,z') = \begin{cases}
\big( z + \ell d_1e_2, z' + \ell(|d_1-d_2|e_0 + d_2e_1) \big) & \text{if $|z| < |z'|$} \\
\big( z + \ell(|d_1-d_2|e_0 + d_2e_1), z' + \ell d_1e_2 \big) & \text{if $|z| > |z'|$} \\
(z,z') & \text{if $|z| = |z'|$ }
\end{cases}$
\end{center}
where $\ell = \big| |z| - |z'| \big|$.  If $\{(z,z')\}$ is the minimal presentation of $M_n$, then $\{\Psi'_n(z,z')\}$ is the minimal presentation of $M_{n+p}$.  
\end{thm}

\section{A shifted family with an unbounded number of relations}
\label{sec:4genlinear}

In this section, we identify a shifted family of affine semigroups $M_n$, each having 4 generators, for which the number of Betti elements (and, therefore, the size of any minimal presentation) grows unbounded as $n \to \infty$ (Theorem~\ref{t:4genlinear}).  

\begin{notation}\label{n:4genlinear}
For the remainder of this section, let 
$$M_n = \<N, N + (1,3), N + (2,1), N + (2,4)\>,$$
and let 
\begin{align*}
A &= (48k + 18, 48k + 28), \qquad B = (18k^2 + 27k + 4, 18k^2 + 27k+4), \qquad \text{and} \\
R_i &= (36k^2 - 18ik + 24k - 9i +3, 36k^2 - 18ik + 36k - 21i +5),
\end{align*}
where $0 \le i \le k$.  
\end{notation}

\begin{thm}\label{t:4genlinear}
If $n = 6k + 1$ with $k \ge 5$, then we have $\Betti(M_n) \supseteq \{A, B, R_0, \dots R_k\}$.  In~particular, $|\Betti(M_n)|$ is unbounded for $n$ large. 
\end{thm}

The proof of Theorem~\ref{t:4genlinear} utilizes the following lemmas.  

\begin{lemma}\label{l:4gena}
If $k \ge 5$, then $\mathsf Z(A) = \{(0,6,2,0), \, (3,0,0,5)\}$.  
\end{lemma}

\begin{proof}
Suppose $(a,b,c,d) \in \mathsf Z(A)$ and let $m = a + b + c + d$.  We first claim $m = 8$.  Examining coordinates, we see
\begin{align}
\label{eq:a1}
48k + 18 &= 6k(a + b + c + d) + (a + 2b + 3c + 3d) = 6mk + (a + 2b + 3c + 3d), \\
\label{eq:a2}
48k + 28 &= 6k(a + b + c + d) + (a + 4b + 5c + 2d) = 6mk + (a + 4b + 2c + 5d),
\end{align}
the first of which we can rearrange to obtain
$$a + 2b + 3c + 3d = 6k(8 - m) + 18.$$
The left side is clearly non-negative, and since $k \ge 5$, the right side is only non-negative if $m \le 8$.  Moreover, if $m \le 7$, then the left side is at most $21$, while the right hand side is at least $48$.  As such, we conclude $m = 8$.  Substituting back into~\eqref{eq:a1} and~\eqref{eq:a2}, we see $(a,b,c,d) \in \mathsf Z(A)$ if and only if
$$18 = a + 2b + 3c + 3d \qquad \text{and} \qquad 28 = a + 4b + 5c + 2d,$$
which are both independent of $k$.  A quick computation for $k = 5$ with \cite{numericalsgpsgap} then yields
$$\mathsf Z(A) = \{(0,6,2,0), \, (3,0,0,5)\}$$
for every $k \ge 5$.  
\end{proof}

\begin{lemma}\label{l:4genb}
Fix $(a,b,c,d) \in \mathsf Z(B)$ and let $m = a + b + c + d$.  If $k \geq 1$, then exactly one of the following hold: 
\begin{enumerate}[(i)]
\item $m = 3k + 3$ and $a = 0$; or 
\item $m = 3k + 4$ and $b = c = d = 0$.  
\end{enumerate}
\end{lemma}

\begin{proof}
By looking at coordinates,
\begin{align}
\label{eq:b1}
18k^2 + 27k + 4 &= 6km + a + 2b + 3c + 3d \qquad \text{and} \\
\label{eq:b2}
18k^2 + 27k + 4 &= 6km + a + 4b + 2c + 5d,
\end{align}
the first of which we can rearrange to obtain
$$a + 2b + 3c + 3d - 4 = 3k(6k + 9 - 2m).$$
The left hand side is clearly non-negative, while the right hand side is only non-negative if $m \le 3k + 4$.  Moreover, \eqref{eq:b1} implies
$$(6k + 3)(3k + 3) - 5 = 18k^2 + 27k + 4 = 6km + a + 2b + 3c + 3d \le 6km + 3m = (6k + 3)m,$$
which is only possible if $m \ge 3k + 3$ since $k \ge 5$.  
Hence, either $m = 3k + 3$ or $m = 3k + 4$. 

Now, supposing $m = 3k+3$, we must show $a = 0$.  Subtracting into~\eqref{eq:b1} and~\eqref{eq:b2}, we obtain $c = 2b + 2d$, and substituting into~\eqref{eq:b1} yields
$$a + 8b + 9d = 9k + 4 = 3m - 5 = 3a + 3b + 3c + 3d - 5 = 3a + 9b + 9d - 5,$$
meaning $5 = 2a + b$.  This forces $a \le 2$, and reducing both sides of
$$6k + 1 = (9k + 4) - (3k + 3) = 5b + 6d$$
modulo $6$ implies $b \equiv 5 \bmod 6$, so $a = 0$.  On the other hand, supposing $m = 3k + 4$, we must show $b = c = d = 0$.  From~\eqref{eq:b1}, we obtain
$$a + 2b + 3c + 3d = 3k + 4 = a + b + c + d.$$
So, $b + 2c + 2d = 0$, at which point non-negativity implies $b = c = d = 0$.
\end{proof}

\begin{lemma}\label{l:4genc}
Suppose $0 \le i \le k$, fix $(a,b,c,d) \in \mathsf Z(R_i)$, and let $m = a + b + c + d$. If $k \geq 2$, then exactly one of the following hold: 
\begin{enumerate}[(i)]
\item $m = 6k - 3i + 1$ and $a = b = 0$; or
\item $m = 6k - 3i + 2$ and $c = d = 0$.
\end{enumerate}
\end{lemma}

\begin{proof}
Examining coordinates, we obtain
\begin{align}
\label{eq:c1}
36k^2 - 18ik + 24k - \phantom{1}9i +3 &= 6km + a + 2b + 3c + 3d \qquad \text{and} \\
\label{eq:c2}
36k^2 - 18ik + 36k - 21i + 5 &= 6km + a + 4b + 2c + 5d.
\end{align}
If $m \ge 6k - 3i + 4$, then~\eqref{eq:c1} implies
$$0 \ge 6k(6k - 3i + 4 - m) = a + 2b + 3c + 3d + 9i - 3 \ge m + 9i - 3 \ge 6k + 6i.$$
which is a contradiction since $k$ is positive.  On the other hand, if $m \le 6k - 3i$, then
$$24k \le 6k(6k - 3i + 4 - m) = a + 2b + 3c + 3d + 9i - 3 \le 3m + 9i - 3 \le 18k + 6,$$
which is a contradiction since $k \ge 2$.  This leaves 3 possible values for $m$.  

First, if $m = 6k - 3i + 3$, then~\eqref{eq:c1} implies
$$-6i = (6k - 9i + 3) - (6k - 3i + 3) = (a + 2b + 3c + 3d) - (a + b + c + d) = b + 2c + 2d,$$
which forces $i = 0$ and $b = c = d = 0$, but this is impossible since $R_0$ is not a multiple of $N$ (the second coordinate is strictly larger than the first).  
Next, supposing $m = 6k - 3i + 1$, from~\eqref{eq:c1} we obtain 
$$0 = 3m - (18k - 9i + 3) = 3m - (a + 2b + 3c + 3d) = 2a + b$$
which implies $a = b = 0$.  Lastly, suppose $m = 6k - 3i + 2$. 
Using~\eqref{eq:c1} and~\eqref{eq:c2},
\begin{align*}
0 &= 3(12k - 9i + 3) - (24k - 21i + 5) - 2(6k - 3i + 2) \\
&= 3(a + 2b + 3c + 3d) - (a + 4b + 2c + 5d) - 2(a + b + c + d) \\
&= 5c + 2d,
\end{align*}
which implies that $c = d = 0$.   
\end{proof}

\begin{proof}[Proof of Theorem~\ref{t:4genlinear}]
After verifying that 
$$(0, 5, 2k + 2, k - 4), \, (3k + 4, 0, 0, 0) \in \mathsf Z(B)$$
and
$$(3i + 1, 6k - 6i + 1, 0,0), \, (0, 0, 2i, 6k - 5i + 1) \in \mathsf Z(R_i),$$
the result follows from Lemmas~\ref{l:4gena}, ~\ref{l:4genb}, and~\ref{l:4genc}.  
\end{proof}

\begin{remark}\label{r:bettiequal}
Based on computational evidence, $\Betti(M_n) = \{A, B, R_0, \ldots, R_k\}$.  However, proving this requires a substantially longer argument that would take us too far astray from the intent of Theorem~\ref{t:4genlinear}.  
\end{remark}

\section{A shifted family with a periodic number of minimal relations}
\label{sec:4genperiodic}

In the final section of this paper, we prove that for each $n \ge 3$, the affine semigroup
$$M_n = \<N, N + (3, 2), N + (4, 3), N + (5, 3)\>$$
has either 2 or 3 minimal relations, demonstrating that the phenomenon identified in Theorem~\ref{t:4genlinear} does not occur for all shifted families of 4-generated affine semigroups.  

\begin{lemma}\label{l:4gentrades}
Suppose $n = 3k + r$ for $k, r \in \ZZ$ with $k \ge 1$ and $0 \le r < 3$.  
The smallest positive values of $c_1$ and $c_2$ for which 
\begin{align*}
c_1(N + (3,2)) &\in \<N, N + (4, 3), N + (5, 3)\> \quad \text{and} \\
c_2(N + (4,3)) &\in \<N, N + (3, 2), N + (5, 3)\>
\end{align*}
are $c_1 = 3$ and $c_2 = 2k + r$, respectively.  
\end{lemma}

\begin{proof}
One can readily check that
$$\pi_n(1, 0, 1, 1) = \pi_n(0, 3, 0, 0) \quad \text{and} \quad \pi_n(k+1, r, 0, k) = \pi_n(0, 0, 2k + r, 0),$$
so $c_1 \le 3$ and $c_2 \le 2k + r$.  
Suppose $(a_0, 0, a_2, a_3) \in \mathsf Z(c_1(N + (3,2)))$.  This yields
\begin{align*}
0 &= n(a_0 - c_1 + a_2 + a_3) - 3c_1 + 4a_2 + 5a_3 \\
0 &= n(a_0 - c_1 + a_2 + a_3) - 2c_1 + 3a_2 + 3a_3
\end{align*}
from which we have $c_1 - a_2 - 2a_3 = 0$, and substituting yields
$$0 = n(a_0 - a_3) + c_1 - 3a_3.$$
Since $c_1 \le 3$, either $a_3 = 0$, which is impossible since it would imply $na_0 + c_1 = 0$, or $a_3 = 1$, in which case $c_1 = 3 - n(a_0 - 1)$ forces $c_1 = 3$ so long as $n \ge 3$.  

Next, suppose $(b_0, b_1, 0, b_3) \in \mathsf Z(c_2(N + (4,3)))$.  Due to the above trade and the minimality of $c_2$, we can assume $b_1 < 3$.  This yields
\begin{align*}
0 &= n(b_0 + b_1 - c_2 + b_3) - 4c_2 + 3b_1 + 5b_3 \\
0 &= n(b_0 + b_1 - c_2 + b_3) - 3c_2 + 2b_1 + 3b_3
\end{align*}
from which we have $c_2 - b_1 - 2b_3 = 0$.  
First, substituting for $c_2$ yields
$$n(b_0 - b_3) = c_1 + b_3,$$
whose positivity implies $b_0 - b_3 \ge 1$.  Moreover, 
$$n(b_0 - b_3) = 4c_2 - 3b_1 - 5b_3 = \tfrac{3}{2}c_2 - \tfrac{1}{2}b_1 \le \tfrac{3}{2}c_2 \le 3k + \tfrac{3}{2}r \le 3(k+1)$$
which forces $b_0 - b_3 = 1$.  From there, reducing both sides of
$$0 = n(b_0 - b_3) - b_1 - 3b_3 = 3k + r - b_1 - 3b_3$$
modulo $3$ implies $b_1 \equiv r \bmod 3$ and thus $b_1 = r$.  Together with
$$2b_3 = c_2 - b_1 \le 2k + r - b_1 = 2k,$$
we obtain
$$c_2 = n(b_0 - b_3) - b_3 = n - b_3 = 3k + r - b_3 \ge 2k + r,$$
as desired.  
\end{proof}

\begin{thm}\label{t:4genperiodic}
Write $n = 3k + r$ for $k, n \in \ZZ_{\ge 0}$ with $r < 3$, and let 
$$R = \big\{ \! \big( (1, 0, 1, 1), (0, 3, 0, 0) \big), \big( (k + 1, r, 0, k), (0, 0, 2k+r, 0) \big) \! \big\}.$$
If $n \ge 3$, then 
$$\rho = \begin{cases}
R & \text{if } r = 0; \\
R \cup \big\{ \! \big( (k + 2, 0, 0, k+1), (0, 3-r, 2k+r-1, 0) \big) \! \big\} & \text{if } r = 1, 2,
\end{cases}$$
is a minimal presentation of $M_n$.  
\end{thm}

\begin{proof}
The two relations in $R$ indeed appear in $\rho$ by Lemma~\ref{l:4gentrades}.  Suppose $(a, b)$ is some relation not generated by $R$.  By first performing the trades in $R$ on $a$ and $b$, it suffices to assume that $a_1, b_1 < 3$ and that $a_2, b_2 < 2k + r$.  There are two cases to consider.  

In the first case, suppose 
$$(a, b) = \big( (a_0, 0, a_2, 0), \, (0, b_1, 0, b_3) \big)$$
This yields the equations
\begin{align*}
0 &= n(a_0 + a_2 - b_1 - b_3) + 4a_2 - 3b_1 - 5b_3 \\
0 &= n(a_0 + a_2 - b_1 - b_3) + 3a_2 - 2b_1 - 3b_3
\end{align*}
from which we obtain $a_2 = b_1 + 2b_3$, yielding
$$0 = n(a_0 + b_3) + b_1 + 3b_3$$
which is impossible since the right hand side is strictly positive.  

For the second case, suppose
$$(a, b) = \big( (0, a_1, a_2, 0), \, (b_0, 0, 0, b_3) \big).$$
This yields the equations
\begin{align*}
0 &= n(a_1 + a_2 - b_0 - b_3) + 3a_1 + 4a_2 - 5b_3 \\
0 &= n(a_1 + a_2 - b_0 - b_3) + 2a_1 + 3a_2 - 3b_3
\end{align*}
from which we obtain $a_1 + a_2 = 2b_3$.  Substituting yields
$$n(b_0 - b_3) = 3a_1 + 4a_2 - 5b_3 = \tfrac{1}{2}a_1 + \tfrac{3}{2}a_2 \le 1 + 3k + \tfrac{3}{2}r \le 3k + 4,$$
which, together with the positivity of $\tfrac{1}{2}a_1 + \tfrac{3}{2}a_2$ implies $b_0 - b_3 = 1$.  Reducing 
$$0 = n(a_1 + a_2 - b_0 - b_3) + 2a_1 + 3a_2 - 3b_3 = n + 2a_1 + 3a_2 - 3b_3$$
modulo $3$ implies $2a_1 \equiv r \bmod 3$.  If $r = 0$, then $a_1 = 0$, meaning $(a, b)$ has the form 
$(a, b) = \big( (0, 0, a_2, 0), \, (b_0, 0, 0, b_3) \big)$,
contradicting the minimality of $c_2$ in Lemma~\ref{l:4gentrades}.  
If,~on the other hand, $r \ne 0$, then $a_1 = 3 - r$.   Substituting one last time into the original equalities, we obtain
$$3k + r = n(b_0 + b_3 - a_1 - a_2) = a_1 + 2a_2 - b_3 = (3 - r) + 2a_2 - b_3$$
which, when combined with
$$2b_3 = a_1 + a_2 = 3 - r + a_2,$$
yields $a_2 = 2k + r - 1$ and $b_3 = k + 1$.  As such, $(a, b)$ is the third claimed relation.  
\end{proof}




\end{document}